\documentclass[a4]{article}
\usepackage{bm}
\usepackage[all]{xy}
\usepackage{graphicx}
\usepackage{verbatim}
\usepackage{wrapfig}
\usepackage{ascmac}
\usepackage{makeidx}
\usepackage{amscd}
\usepackage{url}
\usepackage{comment}
\usepackage{enumerate}
\usepackage{here}
\usepackage{latexsym}
\usepackage{array}
\usepackage{booktabs}
\usepackage{multirow}
\usepackage{mathrsfs}
\usepackage{geometry}
\usepackage{fancyhdr}
\renewcommand{\title}[1]{
\begin{center} \Large \bf #1 \end{center}
}

\renewcommand{\author}[2]{
 \begin{center} #1  \vspace{3mm} \\
  #2 \\
 \end{center}
\addvspace{\baselineskip}
}

\usepackage{amssymb}
\usepackage{amsmath}

\usepackage{amsthm}
\newtheorem{theorem}{Theorem}[section]
\newtheorem{mainthm}[theorem]{Main Theorem}
\newtheorem{proposition}[theorem]{Proposition}
\newtheorem{corollary}[theorem]{Corollary}
\newtheorem{lemma}[theorem]{Lemma}
\newtheorem{example}[theorem]{Example}
\theoremstyle{definition}
\newtheorem{definition}[theorem]{Definition}

\theoremstyle{remark}
\newtheorem*{rem}{Remark}
\makeatletter
\@addtoreset{equation}{section}

\makeatother
\begin{document}
\baselineskip 5mm
\title{(Co)Homology Groups and Categorified Eigenvalue}
\author{${}^1$ Jumpei Gohara ${}^2$ Yuji Hirota ${}^1$ Keisui Ino and~ ${}^1$ Akifumi Sako}
{
${}^1$  Tokyo University of Science,\\ 1-3 Kagurazaka, Shinjuku-ku, Tokyo, 162-8601, Japan\\
${}^2$
Azabu University,\\ 1-17-71 Fuchinobe, Chuo-ku, Sagamihara, Kanagawa, 252-5201, Japan}
\noindent
%{\bf MSC 2010:} 53D55 , 81R60 
\vspace{1cm}

\abstract{

We discuss the relationship between (co)homology groups and categorical diagonalization. We consider the category of chain complexes in the category of finitely generated free modules on a commutative ring. For a fixed chain complex with zero maps as an object, a chain map from the object to another chain complex is defined, and the chain map introduce a mapping cone. 
%%%%%%%%%%%%%%%%%%%%
We found that the fixed object is isomorphic to the (co)homology groups of 
the codomain of the chain map if and only if the chain map is injective to the cokernel of differentials of the codomain chain complex and
the mapping cone is homotopy equivalent to zero.
%%%%%%%%%%%%%%%%%%%%
On the other hand, the fixed object is regarded as a categorified eigenvalue of the chain complex in the context of the categorical diagonalization
introduced by B.Elias and M. Hogancamp arXiv:1801.00191v1. 
It is found that 
(co)homology groups are constructed as the eigenvalue of a chain complex.

}

%\keywords{Categorical diagonalization \and Monoidal homotopy category \and (Co)homology group}

%
%
%%%%%%%%%%%%%%%%%%%%%%%%%%%%%%%%%%%%%%%%%%%%%%%%%%%%%%%%%%%%%%%%%%%%%%%%%%%%%%%%
\section{Introduction}
We give a perspective of a sequence of homology groups by using the categorical diagonalization\cite{diagonalization2,diagonalization1,diagonalization3}. 
We show that categorical sequence of homology groups is given as a categorified eigenvalue of another sequence of chain complexes. We denote the categorified eigenvalue by $\lambda\in ob(\mathscr{A})$ that satisfies homotopy equivalence
\begin{align}\label{ev}
\lambda\otimes V\sim F\otimes V %\quad (for~\lambda , F\in ob(\mathscr{A})~and~V\in ob(\mathscr{V}))
\end{align}
for $F\in ob(\mathscr{A})$ and $V\in ob(\mathscr{V})$. Here $\mathscr{V}$ is an additive category and $\mathscr{A}$ is a monoidal homotopy category 
%%%%%12/17/
defined in Definition 1.6 in \cite{diagonalization1}
%%%%%
which acts on $\mathscr{V}$. The additive category means that a pre-additive category equipped with a biproduct and a zero object, and the monoidal category means that a category equipped with ``tensor product" (See \cite{limit2} for details). $\lambda$ is a scalar object. The scalar object is an object that behaves symmetrically or commutatively for the ``tensor product" (See Section $6.1$ of \cite{diagonalization1} for details).

%%%%%%%%

%\begin{definition}[\cite{diagonalization1,diagonalization2}] \label{defi3}Let $%\mathscr{V}$ be an additive category, 
%$\mathscr{A}$ be a monoidal homotopy category acting on $\mathscr{V}$, and 
%$\lambda $ be a scalar object. 
For a complex $F \in ob(\mathscr{A})$ and
a scalar object $\lambda \in ob(\mathscr{A})$,
a morphism $\alpha:\lambda \to F$ in $\mathscr{A}$ exists, and if
\begin{align*}
\alpha \otimes id_{V}: \lambda \otimes V \longrightarrow F \otimes V 
\end{align*}
gives homotopy equivalent $\lambda\otimes V\sim F\otimes V~(0\neq V\in ob(\mathscr{V}))$ in $\mathscr{V}$, then $V$ is called an eigenobject with an eigenmap $\alpha$ or simply an eigenobject \cite{diagonalization1,diagonalization2}.
%\end{definition}
%%%%%%%%%%%
%In this paper, the scalar object $\lambda $ satisfying $(\ref{ev})$ is called the categorified eigenvalue. 
%%%%%%%%%%%%%%%%%%%
B.Elias and M. Hogancamp showed that
the condition for $V$ to be an eigenobject is expressed using Cone$(\alpha )$ from the following proposition.
\begin{proposition}[\cite{diagonalization1,{diag_book},diagonalization2}]\label{prop1}  
Let $F,\lambda, \mathscr{A},\mathscr{V},\alpha $ be those given above.
%appearing in Definition $\ref{defi3}$. 
Then $V\in ob(\mathscr{V})$ is an eigenobject if and only if {\rm Cone}$(\alpha)\otimes V\sim 0~(V\neq 0)$.
\end{proposition}
\bigskip

%From this proposition, Theorem $\ref{maintheo}$ can be interpreted as follows.

We describe the relation between (co)homology groups and categorified eigenvalues.
The goal of this article is to show the following theorem.
%(Its notations are given in the following sections.)

\begin{mainthm}\label{thm_main}
Let $R$ be a commutative ring with characteristic $0$ and ${}_R\mathcal{M}$ is the category of finitely generated free modules over $R$. 
For the category $\mathcal{C}({}_R\mathcal{M})$ of chain complexes in ${}_R\mathcal{M}$, 
we suppose that
an object $F=(F_\bullet,\,\tilde{d}_\bullet^F)$ and a scalar object $\lambda=(\lambda_\bullet,0)$ in $\mathcal{C}({}_R\mathcal{M})$ are given as
\begin{equation*}
F= ( \cdots \rightarrow  F_{n-1} \xrightarrow{\tilde{d}_{n-1}^F} F_n  \xrightarrow{\tilde{d}_n^F} F_{n+1}  \rightarrow \cdots )
\end{equation*}
and
\[
\lambda= (\cdots \rightarrow  \lambda_{n-1} \xrightarrow{~0~} \lambda_n \xrightarrow{~0~} \lambda_{n+1} \rightarrow \cdots).
\] 
Let
$G_n$ be the submodule of $F_n$ given by $F_n=G_n\oplus {\rm Im} \tilde{d}^F_{n-1}$. 
Let $\alpha=\{\alpha_n:\lambda_n \to F_n\}$ be a morphism in $\mathcal{C}({}_R\mathcal{M})(\lambda,\,F)$. 
%If an eigenobject $V$ is $R$, 
%Then $\lambda$ is a sequence of cohomology groups of $R$ coefficients: 
%\begin{align*}
%\lambda \simeq 
%\left(\cdots \rightarrow  H_{n-1}(F) \xrightarrow{~0~} H_n(F) \xrightarrow{~0~}% H_{n+1}(F) \rightarrow \cdots \right),
%\end{align*}
Then, $\lambda_n$ and the cohomology group $H^n(F)$ with coefficients in $R$ are isomorphic:
\begin{align*}
\lambda _n \simeq \frac{\ker \tilde{d}_{n}^F}{{\rm Im}~\tilde{d}_{n-1}^F}=H^n(F)
\end{align*}
for all $n$,
and $\alpha_n$ is an injection into $G_n$,
if and only if 
$\lambda \otimes R \sim F \otimes R$, that is to say 
$R$ and $\lambda$ are
an eigenobject and
 a categorified eigenvalue, respectively. 

%%%%%%%%%%%%%%%%%%%%%%%%%%%%%
\end{mainthm}
%%%%%%%%%%%%%%
\section{Preliminaries}

We give a quick review of category theory needed for the subsequent discussion~(see for example \cite{homological} and references therein). 
Given a category $\mathscr{C}$, ${ob}( \mathscr{C})$ denotes objects in $\mathscr{C}$, and $\mathscr{C}(X,Y)$ does morphisms from $X$ to $Y$ in $\mathscr{C}$. 

Let $\mathscr{B}$ be an additive category. 
Recall that a chain complex $( X_\bullet , d )$ in $\mathscr{B}$ is a sequence of objects $X_\bullet$ in $\mathscr{B}$ and morphisms $d_n\in \mathscr{B}(X_n,X_{n+1})$
\[
X=(\cdots \rightarrow  X_{n-1} \xrightarrow{d_{n-1}} X_n  \xrightarrow{d_n} X_{n+1}  \rightarrow \cdots)
\]
with $d_{n+1}\circ d_n = 0$ for all $n\in \mathbb{Z}$. They form a category whose objects are all chain complexes in $\mathscr{B}$ and whose morphisms are all chain maps. 
We denoted the category by $\mathcal{C}(\mathscr{B})$. 
\medskip 

\begin{definition}\label{homotopy}
For an additive category $\mathscr{B}$ and its category of chain complexes $\mathcal{C}(\mathscr{B})$ , 
let $X=\{X_i,d_X^i\}$ and $Y=\{Y_i,d_Y^i\}$ be (co)chain complexes in $\mathcal{C}(\mathscr{B})$ and $f$ and $g$ be morphisms of (co)chain complexes in 
$\mathcal{C}(\mathscr{B})(X,Y)$. 
A homotopy $\Psi$ from $f$ to $g$ is a sequence of morphisms $\Psi =\{\Psi^i:X_i\to Y_{i-1}\}$ that satisfies 
\begin{align*}
f^i-g^i=d^{i-1}_Y\circ \Psi^i+\Psi^{i+1}\circ d^i_X
\end{align*}
for all $i\in \mathbb{Z}$.
\begin{align*}
\xymatrix{
 \cdots \ar[r] & X_i \ar[r]^{d^{i}_{X}} \ar[ld]_-{\Psi^{i}} \ar[d]^(0.4){f^i - g^i } & X_{i+1} \ar[r] \ar[ld]^{\Psi^{i+1}} & \\
 Y_{i-1} \ar[r]_{d^{i-1}_{Y}} & Y_i \ar[r] & \cdots
 }
\end{align*} 
If there is a homotopy from $f$ to $g$, it is represented by $f\sim g$ or $f\underset{\Psi}{\sim} g$.

\end{definition}
For this homotopy, homotopy equivalent is defined as follows.

\begin{definition}\label{homotopy equivalent}
If (co)chain complexes $X$ and $Y\in {\rm ob}(\mathcal{C}(\mathscr{B}))$ are isomorphic in a homotopy category $K(\mathscr{B})$, 
then $X$ and $Y$ are said to be homotopy equivalent, denoted by $X\sim Y$. That is, there exist homotopies $\Psi$ and $\Phi$ for 
$f\in \mathcal{C}(\mathscr{B})(X,Y)$ and $g\in \mathcal{C}(\mathscr{B})(Y,X)$ which satisfy the condition 
%homotopies $\Psi$ and $\Phi$ exist for $f\in C(\mathscr{A})(X,Y)$ and $g\in C(\mathscr{C})(Y,X)$, and the condition

\begin{align*}
g\circ f\underset{\Psi}{\sim} id_X,\quad f\circ g\underset{\Phi}{\sim}id_Y.
\end{align*}

\end{definition}
\medskip 

We shall often use the following notation to express a morphism in an additive category $\mathscr{B}$. 
Let $X=X_1\oplus X_2\oplus X_3$ and $Y_1\oplus Y_2\oplus Y_3$ be two direct sums in $\mathscr{B}$. Given a morphism $f\in \mathscr{B}(X,Y)$, 
we define a morphism $f_{ij} \in \mathscr{B}(X_j,Y_i)$ as a composition
\[
f_{ij}:= (X_j\hookrightarrow X\overset{f}{\longrightarrow} Y\longrightarrow Y_i )\qquad (i,j=1,2,3)
\]
of the natural injection $X_j\hookrightarrow X$ and the projection to the $i$-th component $Y\to Y_i$. Using those $f_{ij}$, we shall write $f$ 
in the matrix form as 
\[
 f = \left[ \begin{array}{@{\,}ccc@{\,}} f_{11}&f_{12}&f_{13} \\ f_{21}&f_{22}&f_{23} \\ f_{31}&f_{32}&f_{33} \end{array} \right].
\]
\medskip 

See Chapter $3$ of $\cite{homological}$ for details on these definitions.
\bigskip
\par

In this paper, the symbol $\sim$ 
and $\simeq$ are used to mean homotopy equivalent
and isomorphic, respectively.
% and a shift functor $[1]$ means $(\lambda[1])_n=\lambda_{n-1}.$
%In this article, we show that for $(\ref{ev})$, if an object $V$ is commutative ring $R$ then the categorified eigenvalue (or the scalar object) $\lambda$ is a sequence of homology groups of $R$ coefficients of a chain complex $F$. This fact gives a new perspective of a sequence of homology groups as categorified eigenvalue (or scalar object). 

%%%%%%%%%%%%%%%%%%%%%%%%%%%%%%%%%%%%%%%%%%%%%%%%%%%%%%%%%%%%%%%%%%%%%%%%%%%%%%%%

\section{Homology Groups as Eigenvalues}\label{sect3}

Let $R$ be a commutative ring, and let ${}_R\mathcal{M}$ denote the category of finitely generated free modules over $R$. 
That is, ${}_R\mathcal{M}$ has finitely generated free $R$-modules as objects, and $R$-homomorphisms as morphisms. 
As is well-known, ${}_R\mathcal{M}$ is a monoidal category. 
\medskip 

Let $F=(F_\bullet, \tilde{d}^F_{n})$ and $\lambda=(\lambda_\bullet,0)$ be objects in $\mathcal{C}({}_R\mathcal{M})$ given by 
\begin{equation*}
F= ( \cdots \rightarrow  F_{n-1} \xrightarrow{\tilde{d}_{n-1}^F} F_n  \xrightarrow{\tilde{d}_n^F} F_{n+1}  \rightarrow \cdots )
\end{equation*}
and
\[
\lambda= (\cdots \rightarrow  \lambda_{n-1} \xrightarrow{~0~} \lambda_n \xrightarrow{~0~} \lambda_{n+1} \rightarrow \cdots),
\]
respectively. 
Here, $0$ stands for a constant map assigning the additive identity $\boldsymbol{0}$ to all elements in each component. 
At each integer $n$, one gets the direct decomposition of $F_n$ by 
\begin{equation}\label{eqn_decomposition}
F_n= G_n \oplus {\rm Im}\,\tilde{d}_{n-1}^F, 
\end{equation}
where $G_n$ is a subspace of $F_n$ isomorphic to the cokernel of $\tilde{d}_{n-1}^F$, ${\rm Coker}~\tilde{d}_{n-1}^F=F_n/{\rm Im}\,\tilde{d}_{n-1}^F$. 

We let $\alpha$ be a chain map from $\lambda$ to $F$ which maps each $\lambda_n$ into $G_n\subset F_n$, 
and denote by ${\rm Cone}(\alpha)$ the mapping cone of $\alpha$. 
Namely, $\alpha$ is a family of $R$-homomorphisms $\alpha=\{\alpha_n:\lambda_n \to F_n\}$ with 
\begin{equation}\label{eqn_condition of f}
 \tilde{d}^F_n\circ \alpha_n=0 \quad (\forall n\in \mathbb{Z})
 \qquad\text{and}\qquad 
 {\rm Im}\,\alpha_n\subset G_n
\end{equation}
and ${\rm Cone}(\alpha)$ is a chain complex 
\[
{\rm Cone}(\alpha)
:=
\lambda[1]\oplus F
:=(\cdots \rightarrow  \lambda_{n}\oplus F_{n-1} \xrightarrow{d^Z_{n-1}} 
  \lambda_{n+1}\oplus F_n  \xrightarrow{d^Z_n} \lambda_{n+2}\oplus F_{n+1} \rightarrow \cdots)
\]
with $R$-homomorphisms $\{d_n^Z:\lambda_{n+1}\oplus F_n\to \lambda_{n+2}\oplus F_{n+1}\}$ by 
\[
d_n^Z(\boldsymbol{a}\oplus \boldsymbol{x}) := \boldsymbol{0}\oplus \bigl(\alpha_{n+1}(\boldsymbol{a})+ d_n^F(\boldsymbol{x})\bigr)
\quad (\boldsymbol{a}\oplus \boldsymbol{x}\in \lambda_{n+1}\oplus F_n).
\]
%A shift functor $[1]$ means $(A[1])_n=A_{n+1}$ for 
%$A\in {ob}(\mathcal{C}(\mathscr{B}))$. 
By using the decomposition (\ref{eqn_decomposition}), ${d_n^Z}$ can be rewritten in the matrix form as 
\[
d_n^Z = \left[
\begin{array}{@{\,}ccc@{\,}}0 & 0 & 0 \\ \alpha_{n+1} & 0 & 0 \\ 0 & \delta_n & 0 \end{array}
\right]
\ \ 
:
\ \ 
\left[
\begin{array}{c}
\lambda_{n+1} \\
G_n\\
{\rm Im} \tilde{d}_{n-1}^F
\end{array}
\right]
\rightarrow 
\left[ 
\begin{array}{c}
\lambda_{n+2} \\
G_{n+1}\\
{\rm Im} \tilde{d}_{n}^F
\end{array}
\right]
.
\]
Here, $\delta_n$ denotes a $R$-homomorphism $\tilde{d}_{n}^F$ restricted to $G_n$, 
\[
\delta_n:= \tilde{d}_{n}^F|_{G_n}: G_n \longrightarrow {\rm Im}\,\tilde{d}_{n}^F \subset F_{n+1} .
\]
%$\delta_n:({\rm Im}\,\tilde{d}_{n-1}^F)^c \to {\rm Im}\,\tilde{d}_{n}^F$ satisfying 
%$\tilde{d}_{n}^F(\boldsymbol{x}^c \oplus \boldsymbol{x}_0)=\delta_n(\boldsymbol{x}^c)$ for any 
%$\boldsymbol{x}^c \oplus \boldsymbol{x}_0\in ({\rm Im}\,\tilde{d}_{n-1}^F)^c \oplus {\rm Im}\,\tilde{d}_{n-1}^F$. 
One can easily check that 
\begin{equation}\label{eqn_kernel}
\ker\,\tilde{d}_n^F = \ker\,\delta_n \oplus {\rm Im}\,\tilde{d}^F_{n-1}
\end{equation}
for each $n$. 
From (\ref{eqn_kernel}) and the first condition in (\ref{eqn_condition of f}), it follows that 
\begin{equation}\label{eqn_subset}
{\rm Im}\,\alpha_n\subset \ker \delta_n. 
\end{equation} 
\medskip 

\begin{lemma}\label{lem1}
If ${\rm Cone}(\alpha)\sim 0$, then 
each $\alpha_n$ is injective, and
$\lambda_n$ and the $n$-th (co)homology groups are isomorphic each other for each $n${\rm :}
\begin{align*}
\lambda _n\simeq H^n(F):=\frac{\ker \tilde{d}_{n}^F}{{\rm Im}\,\tilde{d}_{n-1}^F}.
\end{align*}
\end{lemma}

\begin{proof}
By the assumption $Z={\rm Cone}(\alpha)\sim 0$, there exists a homotopy $\{\Psi^n:Z_n\to Z_{n-1}\}$ from $0$ to ${id}_{Z}$
as follows: 
\begin{align*}
\xymatrix{
 \cdots \ar[r] &  \lambda_n\oplus F_{n-1} \ar[r]^{d_{n-1}^Z} &  \lambda_{n+1}\oplus F_n \ar[r]^-{d_{n}^Z}  \ar[r] \ar[ldd]^{\Psi^{n}} \ar@{-}[d]& \lambda_{n+2}\oplus  F_{n+1}  \ar[ldd]^{\Psi^{n+1}}  \ar[r] & \cdots \\
 \cdots \ar[r] & 0 \ar[r] & 0 \ar[r]  \ar[d]^{id} & 0 \ar[r]  & \cdots \\
\cdots \ar[r] &  \lambda_n \oplus F_{n-1} \ar[r]^{d_{n-1}^Z} &  \lambda_{n+1}\oplus F_{n}  \ar[r]^-{d_{n}^Z}  \ar[r]  &   \lambda_{n+2}\oplus F_{n+1}  \ar[r] & \cdots \\
  }
\end{align*}
Let us express $\Psi^n:\lambda_{n+1}\oplus \bigl(G_n\oplus {\rm Im}\,\tilde{d}_{n-1}^F\bigr)
\rightarrow 
\lambda_{n}\oplus \bigl(G_{n-1}\oplus {\rm Im}\,\tilde{d}_{n-2}^F\bigr)
$ in the matrix form as
\begin{align*}
  \Psi^n=\left[\begin{array}{@{\,}ccc@{\,}}
  \psi_{11}^n&\psi_{12}^n&\psi_{13}^n\\
  \psi_{21}^n&\psi_{22}^n&\psi_{23}^n\\
  \psi_{31}^n&\psi_{32}^n&\psi_{33}^n
  \end{array}\right]
\ \ 
:
\ \ 
\left[ 
\begin{array}{c}
\lambda_{n+1} \\
G_n\\
{\rm Im} \tilde{d}_{n-1}^F
\end{array}
\right]
\rightarrow 
\left[ 
\begin{array}{c}
\lambda_{n} \\
G_{n-1}\\
{\rm Im} \tilde{d}_{n-2}^F
\end{array}
\right] .
\end{align*} 
Here, $\{\psi_{\bullet 1}^n\},\,\{\psi_{\bullet 2}^n\}$ and $\{\psi_{\bullet 3}^n\}$ stand for homomorphisms defined on $\lambda_{n+1},\,G_n$ and 
${\rm Im}\,\tilde{d}_{n-1}^F$ respectively. 
From Definition $\ref{homotopy equivalent}$, $\Psi$ satisfies
\begin{align*}
  -id_Z&=d_{n-1}^Z\circ\Psi^{n} + \Psi^{n+1}\circ d_{n}^Z \notag \\
  &= \left[
  \begin{array}{@{\,}ccc@{\,}}
  0&0&0\\
 \alpha_{n}\circ\psi_{11}^n & \alpha_{n}\circ\psi_{12}^n & \alpha_{n}\circ\psi_{13}^n \\
  \delta_{n-1}\circ\psi_{21}^n & \delta_{n-1}\circ\psi_{22}^n& \delta_{n-1}\circ\psi_{23}^n 
  \end{array}
  \right]
  +\left[
  \begin{array}{@{\,}ccc@{\,}}
   \psi_{12}^{n+1}\circ \alpha_{n+1}  & \psi_{13}^{n+1}\circ \delta_{n}&0 \\
   \psi_{22}^{n+1}\circ \alpha_{n+1}  & \psi_{23}^{n+1}\circ \delta_{n}&0 \\
   \psi_{32}^{n+1}\circ \alpha_{n+1}  & \psi_{33}^{n+1}\circ \delta_{n}&0
  \end{array}
  \right].%\label{eq.00}
  \end{align*}
From the diagonal entries in the above equation, we have
\begin{align}
\psi_{12}^{n+1}\circ \alpha_{n+1}&=-id_{\lambda_{n+1}}, \label{eq.1}\\
 \alpha_{n}\circ \psi_{12}^n+ \psi_{23}^{n+1}\circ \delta_{n}&=-id_{G_n},\label{eq.2}\\
 \delta_{n-1}\circ\psi_{23}^n &=-id_{{\rm Im}\,\tilde{d}^F_{n-1}}.\label{eq.3}
\end{align}
It follows from $(\ref{eq.1})$ that  
\begin{equation}\label{eqn_rank calculation1}
{\rm dim}\,\lambda_{n+1} = {\rm rank}\,(\psi_{12}^{n+1}\circ \alpha_{n+1}) \leqq {\rm rank}\,\alpha_{n+1}.
\end{equation}
On the other hand, by the rank nullity theorem for $\alpha_{n+1}:\lambda_{n+1}\to F_{n+1}$, we have 
\begin{equation}\label{eqn_rank calculation2}
{\rm rank}\,\alpha_{n+1} = {\rm dim}\,\lambda_{n+1} - {\rm null}\,\alpha_{n+1}. 
\end{equation}
From (\ref{eqn_rank calculation1}) and (\ref{eqn_rank calculation2}), we get ${\rm null}\,\alpha_{n+1}=0$ to find that $\lambda_n$ is isomorphic 
to ${\rm Im}\,\alpha_n$ at each $n$. 
\medskip 

It follows from (\ref{eq.2}) that 
\begin{equation}\label{eqn_rank calculation3}
{\rm rank}\,G_n = {\rm rank}\,(\alpha_{n}\circ \psi_{12}^n+ \psi_{23}^{n+1}\circ \delta_{n})
\leqq {\rm rank}\,\alpha_n + {\rm rank}\,\delta_n.
%& \leqq {\rm rank}\,(\alpha_{n}\circ \psi_{12}^n) + {\rm rank}\,(\psi_{23}^{n+1}\circ \delta_{n})\\
\end{equation}
By applying the rank nullity theorem to $\delta_n:G_n \to {\rm Im}\,\tilde{d}_{n}^F$, 
\begin{equation}\label{eqn_rank calculation4}
{\rm rank}\,G_n = {\rm null}\,\delta_n + {\rm rank}\,\delta_n.
\end{equation}
From (\ref{eqn_rank calculation3}), (\ref{eqn_rank calculation4}) and (\ref{eqn_subset}), it turns out that ${\rm Im}\,\alpha_n = \ker \delta_n$. 
\medskip 

The condition (\ref{eqn_kernel}) yields to that 
\[
H^n(F)=\ker\tilde{d}^F_n/{\rm Im}\,\tilde{d}^F_{n-1}\simeq \ker\delta_n.
\]
by the fundamental homomorphism theorem. Consequently, we have 
\[
\lambda_n \simeq {\rm Im}\,\alpha_n = \ker \delta_n \simeq H^n(F), 
\]
which shows the proposition. 
\end{proof}
\bigskip

We find that ${\rm Cone}(\alpha)\sim 0$ is sufficient for each $\lambda_n$ to be isomorphic to their (co)homology group $H^n(F)$. 
Now, what conditions are needed for ${\rm Cone}(\alpha)\sim 0$ ? The answer to the query is given as follows:
\medskip 

\begin{lemma}\label{lem2}
Assume that each $\alpha_n$ is an injection into $G_n\subset F_n$. 
If each $\lambda_n$ is isomorphic to $H^n(F)$, 
then the mapping cone of $\alpha$ is null-homotopic, ${\rm Cone}(\alpha)\sim 0$. 
\end{lemma}

\begin{proof}
Since $\lambda_n\simeq H^n(F)$, one has $\lambda_n\simeq \ker\delta_n$ at each $n$. Therefore, $\alpha_n$ are thought of as isomorphisms 
$\alpha_n:\ker\delta_n\to \ker\delta_n$ by the assumption. 
$\ker\delta_n$ is a finitely generated free submodule of $G_n$. So, $G_n$ can be decomposed to two submodules $\ker\delta_n$ and some $K_n$ as 
\[
G_n = K_n\oplus \ker\delta_n
\]
at each $n$. 
Remark that each map $\delta_n$ restricted to $K_n$ is an isomorphism between $K_n$ and ${\rm Im}\,\tilde{d}^F_{n}$.

Define two maps $\varphi^n_{1}:G_n \to \lambda_n \simeq \ker\delta_n$ and $\varphi^n_{2}:{\rm Im}\,\tilde{d}^F_{n-1} \to G_{n-1}$ as 
\[
\varphi^n_{1}(\boldsymbol{x}) = 
\begin{cases}
\, -\alpha_n^{-1}(\boldsymbol{x})&\quad (\text{iff $\boldsymbol{x}\in \ker\delta_n$}) \\
\, \boldsymbol{0}&\quad (\text{iff $\boldsymbol{x}\in K_n$})
\end{cases},  
\]
and
\[
\varphi^n_{2}(\boldsymbol{x}) := -(\delta_{n-1}|_{K_{n-1}})^{-1}(\boldsymbol{x})\quad \text{for any $\boldsymbol{x}\in {\rm Im}\,\tilde{d}^F_{n-1}$}
\]
respectively. 

Using $\varphi^n_{1}$ and $\varphi^n_{2}$, we construct a family of morphisms $\{\Phi^n:\lambda_{n+1}\oplus F_n\to \lambda_n\oplus F_{n-1}\}_n$ by
\begin{align}\label{eq6}
  \Phi^n=\left[
  \begin{array}{@{\,}ccc@{\,}}
   0& \varphi_{1}^n & 0 \\
   0& 0 & \varphi_{2}^n \\
   0& 0 & 0 
  \end{array}
  \right]
\ \ 
:
\ \ 
\left[ 
\begin{array}{c}
\lambda_{n+1} \\
G_n\\
{\rm Im} \tilde{d}_{n-1}^F
\end{array}
\right]
\rightarrow 
\left[ 
\begin{array}{c}
\lambda_{n} \\
G_{n-1}\\
{\rm Im} \tilde{d}_{n-2}^F
\end{array}
\right] 
.
\end{align}
Then, 
\begin{align*}
&d_{n-1}^Z\circ \Phi^{n} + \Phi^{n+1}\circ d_{n}^Z \\
=&\ \left[\begin{array}{@{\,}ccc@{\,}} 0&0&0 \\ \alpha_n&0&0 \\ 0&\delta_{n-1}&0 \end{array}\right]
  \left[\begin{array}{@{\,}ccc@{\,}} 0& \varphi_{1}^n & 0 \\ 0& 0 & \varphi_{2}^n \\ 0& 0 & 0 \end{array}\right] + 
 \left[\begin{array}{@{\,}ccc@{\,}} 0& \varphi_{1}^{n+1} & 0 \\ 0& 0 & \varphi_{2}^{n+1} \\ 0& 0 & 0 \end{array}\right]
 \left[\begin{array}{@{\,}ccc@{\,}} 0&0&0 \\ \alpha_{n+1}&0&0 \\ 0&\delta_{n}&0 \end{array}\right]\\
=& \left[\begin{array}{@{\,}ccc@{\,}} \varphi_1^{n+1}\circ \alpha_{n+1}&0&0 \\ 0&\alpha_n\circ\varphi_1^n + \varphi_2^{n+1}\circ \delta_n&0 \\ 
  0&0&\delta_{n-1}\circ \varphi_2^n\end{array}\right].
\end{align*}

We calculate $(\varphi_1^{n+1}\circ \alpha_{n+1})(\boldsymbol{a}),\,(\alpha_n\circ\varphi_1^n + \varphi_2^{n+1}\circ \delta_n)(\boldsymbol{b})$ and 
$(\delta_{n-1}\circ \varphi_2^n)(\boldsymbol{c})$ for any 
$\boldsymbol{a}\oplus \boldsymbol{b}\oplus \boldsymbol{c}\in \lambda_{n+1}\oplus G_n\oplus {\rm Im}\,\tilde{d}^F_{n-1}$ to show that 
$d_{n-1}^Z \circ \Psi^{n} + \Psi^{n+1}\circ d_{n}^Z=-id_Z$.

It is easy to check $\varphi_1^{n+1}\circ \alpha_{n+1}=-id$. 
The image $\alpha_{n+1}(\boldsymbol{a})$ is also an element in $\ker\delta_{n+1}$
(see (\ref{eqn_subset})). Therefore, we have  
\[
(\varphi_1^{n+1}\circ \alpha_{n+1})(\boldsymbol{a})=-\alpha_{n+1}^{-1}\bigl(\alpha_{n+1}(\boldsymbol{a})\bigr)=-\boldsymbol{a}.
\] 
Next, we show $\alpha_n\circ\varphi_1^n + \varphi_2^{n+1}\circ \delta_n=-id$. The element $\boldsymbol{b}$ is expressed as 
$\boldsymbol{b}=\boldsymbol{b}_1\oplus\boldsymbol{b}_2$ where $\boldsymbol{b}_1\in K_n$ and $\boldsymbol{b}_2\in \ker\delta_n$. 
Then, 
\[(\alpha_n\circ\varphi_1^n)(\boldsymbol{b})=\alpha_n\bigl(-\alpha_n^{-1}(\boldsymbol{b}_2)\bigr)=-\boldsymbol{b}_2
\] 
and moreover, 
\[
(\varphi_2^{n+1}\circ \delta_n)(\boldsymbol{b})=-\delta_{n}^{-1}\bigl(\delta_{n}(\boldsymbol{b}_1)\bigr)=-\boldsymbol{b}_1.
\] 
Accordingly, we have $(\alpha_n\circ\varphi_1^n + \varphi_2^{n+1}\circ \delta_n)(\boldsymbol{b})=-\boldsymbol{b}$ for $\boldsymbol{b}\in G_n$. 

Lastly, we have $(\delta_{n-1}\circ \varphi_2^n)(\boldsymbol{c})=\delta_{n-1}\bigl(-\delta_{n-1}^{-1}(\boldsymbol{c})\bigr)=-\boldsymbol{c}$ 
for $\boldsymbol{c}\in {\rm Im}\,\tilde{d}^F_{n-1}$.

From the above computation, we get $d_{n-1}^Z \circ\Psi^{n} + \Psi^{n+1}\circ d_{n}^Z=-id_Z$ which shows that ${\rm Cone}(\alpha)$ is null-homotopic. 
\end{proof}

Combining Lemma \ref{lem1} and Lemma \ref{lem2}, we get the following result.
\medskip 

\begin{theorem}\label{maintheo}
Let $F=(F_\bullet,\,\tilde{d}_\bullet^F)$ and $\lambda=(\lambda_\bullet,0)$ be objects in $\mathcal{C}({}_R\mathcal{M})$. 
Let $G_n$ be the submodule of $F_n$ given by 
$F_n = G_n \oplus {\rm Im} \tilde{d}_{n-1}^F$
%(\ref{eqn_decomposition}) 
and 
$\alpha=\{\alpha_n:\lambda_n \to F_n\}$ a morphism in $\mathcal{C}({}_R\mathcal{M})(\lambda,\,F)$ satisfying ${\rm Im}\,\alpha_n\subset G_n$ at each $n\in \mathbb{Z}$. 
Then, each $\lambda_n$ is isomorphic to $H^n(F)$ 
and $\alpha_n$ is an injection into $G_n$,
if and only if ${\rm Cone}(\alpha)\sim 0$. 
\end{theorem}
\bigskip

Now, 
we prove Main Theorem \ref{thm_main},
by combining Theorem \ref{maintheo} with Proposition \ref{prop1}.
\begin{proof}
Proposition \ref{prop1} shows that  if ${\rm Cone}(\alpha )\otimes V \sim 0$ then $\lambda$
is a categorified eigenvalue.
According to the supposition  of the theorem,
the eigenobject $V$ is $R$, then
it is trivially ${\rm Cone}(\alpha )\otimes R = {\rm Cone}(\alpha )\sim 0$  
since ${}_R\mathcal{M}$ is the category of modules over $R$.
In this case  each $\lambda_n$  is isomorphic to the cohomology $H^n (F)$ from Theorem \ref{maintheo}.

Conversely, if each $\lambda_n$  is isomorphic to the cohomology $H^n (F)$ 
and each $\alpha_n$ is an injection into $G_n$,
then ${\rm Cone}(\alpha )\otimes R \sim 0$ from Theorem \ref{maintheo}.
Proposition \ref{prop1} shows that 
$R$ and $\lambda$ are
an eigenobject and
 a categorified eigenvalue, respectively.
\end{proof}

\begin{corollary}
An arbitrary sequence of (co)homology groups $H^\bullet (F_\bullet,\,\tilde{d}_\bullet^F)$ 
with coefficients in $R$ 
is a categorified eigenvalue of $F$ with its eigenobject $R$.
\end{corollary}

\begin{proof}
We can construct an injection $\alpha_n : H^n(F) \to G_{n}$
to assign a representative $x$ in $F_n$ to each element 
$[x]$ in $H^n(F)$.
Therefore,
it follows from Main Theorem \ref{thm_main} that 
$H^\bullet (F)$ is a categorified eigenvalue of $F$.
\end{proof}

\medskip 

\begin{rem} 
Similarly, a proposition that reverses subscripts of the (co)chain complex also holds.
\end{rem}

We give a description of Theorem \ref{maintheo} by terminologies of homology to avoid confusion caused by the difference of indexes, in the matrix representation version
Lemma \ref{App_lem1} in Appendix \label{A}. 
%\subsection{cor1}
\begin{corollary}\label{cor1} 
Let $\mathcal{C}({}_R\mathcal{M})$ be a category of chain complexes of 
finitely generated free modules over $R$.
Fix a chain complex $F \in ob(\mathcal{C}({}_R\mathcal{M}))$ with boundary operators $\tilde{\partial } _n$ as 
\begin{align*}
F=( \cdots \rightarrow  C_{n+1} \xrightarrow{\tilde{\partial}_{n+1}} C_n  \xrightarrow{\tilde{\partial}_n} C_{n-1}  \rightarrow \cdots ), \quad \tilde{\partial}_{n}\tilde{\partial}_{n+1}=0.
\end{align*}
We suppose that
a scalar object $\lambda \in ob(\mathcal{C}({}_R\mathcal{M}))$ is given by
\begin{align*}
\lambda&=(\cdots \rightarrow  \lambda_{n+1} \xrightarrow{~0~} \lambda_n \xrightarrow{~0~} \lambda_{n-1} \rightarrow \cdots).
\end{align*}
Under the direct decomposition of 
\begin{align*}
C_n=
({\rm Im}~\tilde{\partial }_{n+1})^c
\oplus {\rm Im}~\tilde{\partial }_{n+1}, 
\end{align*}
we consider a morphism $\alpha:=\{\alpha_n:\lambda_n\to ({\rm Im}~\tilde{\partial } _{n+1})^c\subset C_n\}\in \mathcal{C}({}_R\mathcal{M})(\lambda,F)$ and $Z:={\rm Cone}(\alpha)=\lambda[-1]\oplus F$ with morphisms $\partial_n^Z:Z_n\to Z_{n-1}$ such that 
\begin{align*}
\partial_n^Z=\left[
\begin{array}{ccc}
0&0&0\\
\alpha_{n-1}&0&0\\
0&\partial_n&0
\end{array}\right]:
\left[\begin{array}{c}
\lambda_{n-1}\\
({\rm Im}~ \tilde{\partial }_{n+1})^c\\
{\rm Im}~ \tilde{\partial }_{n+1}
\end{array}\right]\to
\left[\begin{array}{c}
\lambda_{n-2}\\
({\rm Im}~ \tilde{\partial }_{n})^c\\
{\rm Im}~ \tilde{\partial }_{n}
\end{array}\right],
\end{align*}
where $\partial_n:({\rm Im}~ \tilde{\partial}_{n+1})^c \to {\rm Im}~ \tilde{\partial}_{n}$, and $\alpha_n$ is a map from $\lambda_n$ to $\ker \tilde{\partial}_{n}~\cap~ ({\rm Im}~\tilde{\partial}_{n+1})^c$. A condition ${\rm Cone}(\alpha)\sim 0$ is satisfied if and only if each $\lambda_n$ is isomorphic to 
a homology group
\begin{align*}
\lambda_n \simeq 
 H_n (F),
\end{align*}
and each $\alpha_n$ is an injection.
\end{corollary}

\begin{example}
Fix a chain complex $F\in ob(\mathcal{C}({}_R\mathcal{M}))$ as a chain complex of $S^1$ with coefficients in $\mathbb Z$:
\begin{align*}
 F= \left( 0  \longrightarrow C_1 (S^1) \xrightarrow{\mbox{} \ \tilde{\partial}_{1}\ \mbox{}} C_0 (S^1)\longrightarrow 0 \right).
\end{align*}
Let $\lambda\in ob(\mathcal{C}({}_R\mathcal{M}))$ be a sequence of 
homology groups given  as
\begin{align*}
\lambda &=\left (0\longrightarrow H_1(S^1)\longrightarrow H_0(S^1)\longrightarrow 0\right)\\
&=\left (0\longrightarrow \mathbb{Z}\xrightarrow{~0~} \mathbb{Z}\longrightarrow 0\right).
\end{align*}
Suppose that a basis of $C_1(S^1)$ is given by three edges $[AB],\,[BC]$ and $[CA]$ of a triangle $\Delta ABC$ and 
a basis of $C_0(S^1)$ is chosen as $[A]$ in $({\rm Im}~ \tilde{\partial }_{1})^c$, ~$[B]-[A]$ and $[C]-[B]$ in ${\rm Im}~ \tilde{\partial }_{1}$. We show that ${\rm Cone}(\alpha )$ and $0$ are homotopy equivalent for a morphism $\alpha:=\{\alpha_n:H_n(S^1)\to C_n(S^1)\}\in \mathcal{C}({}_R\mathcal{M})(\lambda, F)$ where 
\begin{align*}
\begin{array}{rccl}
\alpha_1: &  n & \mapsto & {}^t(n,n,n)=n[AB]+n[BC]+n[CA], \\
\alpha_0: &  n & \mapsto & {}^t(n,0,0)=n[A].
\end{array}
\end{align*}

%the basis of $C_0(S^1)$ are given by three vertices in a triangulation of $S^1$. For the following diagram,
\begin{align*}
\xymatrix{
0 \ar[r]  & \mathbb{Z}   \ar@{-}[d] \ar[r]^-{ \partial_2^Z }  & \mathbb{Z} \oplus C_1 (S^1) \ar[r]^-{\partial_1^Z} \ar@{-}[d] \ar[ldd]^{\Psi^1} & C_0 (S^1)  \ar[r] \ar@{-}[d]  \ar[ldd]^{\Psi^0} & 0 \\
0 \ar[r]  & 0 \ar[r] \ar[d]^{id}  & 0 \ar[r] \ar[d]^{id}   & 0 \ar[r] \ar[d]^{id}    & 0 \\
0 \ar[r] & \mathbb{Z}   \ar[r]_-{ \partial^Z_2} &  \mathbb{Z} \oplus C_1 (S^1)  \ar[r]_-{\partial_1^Z} & C_0 (S^1)  \ar[r] & 0
}
\end{align*}
 Then morphisms and homotopies are given by
\begin{align*}
\begin{array}{rccl}
\partial_2^Z \colon &  n & \mapsto & {}^t( 0,n,n,n), \\
 \partial_1^Z \colon & {}^t( k , l , m , n) &  \mapsto & {}^t(k,l-n,m-n), \\
\Psi^1 \colon & {}^t( k,l,m,n) &  \mapsto & -n,  \\
\Psi^0 \colon &  {}^t(l,m,n) & \mapsto &  {}^t( -l, -m , -n,0)
 \end{array}
\end{align*}
for $k,l,m,n\in \mathbb{Z}$, that is,
\begin{align*}
\partial_2^Z  ={}^t\left[ 0 , 1 ,1, 1 \right] &,\quad
 \partial_1^Z= \left[
 \begin{array}{cccc}
 1&0 & 0 & 0 \\
  0&1 & 0 & -1  \\
 0&0 & 1 & -1  
 \end{array}
 \right] , \\
 \Psi^1 = \left[ 0,0,0,-1 \right] &,\quad
 \Psi^0 =
 \left[
 \begin{array}{ccc}
 -1   & 0 & 0 \\
 0  & -1  & 0 \\
  0&0&-1\\
 0  &  0 & 0 
  \end{array}
 \right].
 \end{align*}
The homotopy equivalent conditions are checked from the following calculation. 
\begin{align*}
 \Psi^{1} \circ  \partial_2^Z  &=-1,\\
 \Psi^0 \circ  \partial_1^Z  +  \partial_2^Z\circ \Psi^1&=\left[ 
 \begin{array}{cccc}
 -1 & 0 & 0 & 0 \\
 0 & -1 & 0 & 0 \\
 0 & 0 & -1 & 0 \\
 0 & 0& 0& -1
 \end{array}
 \right] ,\\
   \partial_1^Z \circ \Psi^{0} &=\left[ 
 \begin{array}{ccc}
 -1 & 0 & 0 \\
 0 & -1 & 0 \\
 0 & 0 & -1 
 \end{array}
 \right].
\end{align*}
Therefore, $Cone (\alpha) \sim 0$ is satisfied.
We can regard the sequence of homology groups  $\lambda$ as the categorified eigenvalue of the chain complex of $S^1$.

\end{example}

%%%%%%%%%%%%%%%%%%%%%%%%%%%%%%%%%%%%%%%%%%%%%%%%%%%%%%%%%%%%%%%%%%%%%%%%%%%%%%%

% % % % % % % % % % % % % %
% % % % % % % % % % % % % %

\section*{Acknowledgements}
\noindent 
A.S.\ was supported by JSPS KAKENHI Grant Number 16K05138.

% % % % % % % % % % % % % % % % % % 
\appendix
\section{Appendix} \label{A}
In this appendix, explicit conditions of homotopy $\Psi$ in Section \ref{sect3}
are studied.
We use the same notation in the proof of Lemma \ref{lem1}.

\bigskip

%\subsection{prf1}
%%%%%%%%%%%%%%%%%%%%%%%%%%%%%%%%%%%%%
%%%%%%%%%%%%%%%%%%%%%%%%%%%%%%%%%%%%%

${\rm Cone}(\alpha)\sim 0$ means that there exist homotopy $\Psi^n$ for all $n$.

To satisfy $(\ref{eq.1})$ and $(\ref{eq.3})$, for  
$G_{n+1} = 
{\rm Im}~\alpha_{n+1} \oplus ({\rm Im}~\alpha_{n+1})^c
$ in $
F_n=G_n \oplus {\rm Im}~\tilde{d}_{n-1}^F 
$, we put
\begin{align*}
\psi^{n+1}_{12}&=\begin{cases}
-\alpha_{n+1}^{-1}& {\rm for}~{\rm Im}~\alpha_{n+1}\\
g_{n+1}& {\rm for}~({\rm Im}~\alpha_{n+1})^c
\end{cases}, \ &: G_{n+1} \rightarrow  \lambda_n\\
\psi_{23}^{n}&=-(\delta_{n-1})^{-1}+c_{n-1}, &: {\rm Im}~ \tilde{d}_{n-1}^F \rightarrow 
G_{n-1}
%{\rm Im}~ \tilde{d}_{n-2}^F
\end{align*}
where 
an isomorphism $\alpha_{n+1}^{-1}$ 
from ${\rm Im}~\alpha_{n+1}$ to $\lambda_n$ and
an isomorphism 
$(\delta_{n-1})^{-1}$
from ${\rm Im}~ \tilde{d}_{n-1}^F \simeq {\rm Im}~ \delta_{n-1}$
to $G_{n-1} -\ker \delta_{n-1}$
are fixed, respectively.

From $(\ref{eq.3})$, a condition
\begin{align}
\delta_{n-1} \circ c_{n-1} = 0,
\quad 
{\rm Im}~c_{n-1}\subset \ker \delta_{n-1} \simeq
\ker \tilde{d}_{n-1}^F\slash {\rm Im}~\tilde{d}_{n-2}^F
\label{A_1_p_0}
\end{align}
is obtained.
%and to satisfy $(\ref{eq.1})$ 
%and ${\rm Im}~(\delta_{n-1})^{-1}$ is determined as 
%${\rm Im}~(\delta_{n-1})^{-1}\subset (\ker d^F_{n-1})^c$. 
%In other words, $(\delta_{n-1})^{-1}$ is an isomorphism between 
%${\rm Im}~\tilde{d}^F_{n-1}$ and $({\rm Im}~\tilde{d}^F_{n-2})^c-\ker 
%d^F_{n-1}$. 
From $(\ref{eq.2})$, 
$
g_{n}:({\rm Im}~\alpha_{n})^c \to\lambda_{n},
$ and $c_n : {\rm Im}~ \tilde{d}_{n}^F \rightarrow G_{n-1}$
satisfy
\begin{align}
%\left\{
%\begin{array}{cc}
c_n \circ \delta_n -id_{{\rm Im}~\alpha_n}=-id_{\ker \delta_n} \qquad
& {\rm for}~{\rm Im}~\alpha_{n}
\label{A_1_p_1}
\\
c_n \circ \delta_n + \alpha_n \circ g_n=-id_{\ker \delta_n} \oplus 
0_{ K_n }
  \qquad
& {\rm for}~({\rm Im}~\alpha_{n})^c .
\label{A_1_p_2}
%\end{array}
%\right.
\end{align}
Here, we decompose 
$G_n$
into $\ker \delta_n \oplus K_n$,
in other words
\begin{align*}
&G_n
= A \oplus B \oplus C \oplus D ,\\
&A= \ker \delta_n \cap ({\rm Im}~\alpha_{n})^c , \qquad
B=  K_n  \cap ({\rm Im}~\alpha_{n})^c ,
\\
&C= \ker \delta_n \cap {\rm Im}~\alpha_{n} ,  \qquad
D= K_n \cap {\rm Im}~\alpha_{n} .
\end{align*}
Recall that 
\begin{align*}
{\rm Im}~\alpha_{n} \subset \ker \delta_n 
\end{align*}
since $\alpha$ is a chain map or
$ d_Z \circ d_Z =0$. Then 
\begin{align}
D= K_n \cap {\rm Im}~\alpha_{n} = \{ \boldsymbol{0}  \},
\label{A_p_1_2.5}
\end{align}
and
(\ref{A_1_p_1}) is equivalent to
\begin{align}
c_n \circ \delta_n |_{{\rm Im}\alpha_n} 
=  0_C , \qquad C =  {\rm Im}~\alpha_n
\label{A_1_p_3}
\end{align}
for  ${\rm Im}~\alpha_n$.
%From (\ref{A_1_p_0}) ,
%$({\rm Im } \delta_n )^c \cap \ker \delta_n \simeq H^n (F)$
%has possibility only nonzero after acting $c_n \circ \delta_n$ in $D$.
%However, for arbitrary nonzero $v \in ({\rm Im } \delta_n )^c \cap \ker \delta_n %$, 
%\begin{align*}
%c_n \circ \delta_n (v) = 0 ,
%\end{align*}
%because $v \in \ker \delta_n$.
%Therefore the only consistent case with (\ref{A_1_p_3})
%is 

Next, let us consider the condition (\ref{A_1_p_2}),
that is equivalent to 
\begin{align}
(c_n \circ \delta_n + \alpha_n \circ g_n )|_{A\oplus B} = 
-id_{A} \oplus 
0_{B}. \label{A_1_p_4}
\end{align}
For arbitrary $\boldsymbol{v} \in A$,
the L.H.S. of (\ref{A_1_p_4}) becomes
\begin{align*}
(c_n \circ \delta_n + \alpha_n \circ g_n )|_{A} (\boldsymbol{v})=
(\alpha_n \circ g_n )|_{A} (\boldsymbol{v})  \ \in {\rm Im}~\alpha_n .
\end{align*}
On the other side, R.H.S. of (\ref{A_1_p_4})  becomes
\begin{align*}
(-id_{A} \oplus 0_{B} ) (\boldsymbol{v}) = - \boldsymbol{v} ,
\end{align*}
since $v \in A$. 
By definition of $A$, $A \cap  {\rm Im}~\alpha_n = \{ \boldsymbol{0} \}$,
therefore it turns out that
\begin{align}
A= \ker \delta_n \cap ({\rm Im}~\alpha_{n})^c = \{ \boldsymbol{0} \}.
\label{A_1_p_5}
\end{align}
For arbitrary $\boldsymbol{w} \in B$, from (\ref{A_1_p_4}) we obtain
\begin{align}
(c_n \circ \delta_n + \alpha_n \circ g_n )|_{B} (\boldsymbol{w}) = \boldsymbol{0}
\end{align}
because $ (-id_{A} \oplus 
0_{B}) (\boldsymbol{w}) = \boldsymbol{0}$.
From (\ref{A_p_1_2.5}) and (\ref{A_1_p_5}),
we find 
\begin{align}
{\rm Im}~\alpha_n = \ker \delta_n , \qquad
 K_n  = ({\rm Im}~\alpha_{n})^c
\label{A_1_p_6}
\end{align}
and 
\begin{align}
G_n
= B \oplus C ,
\end{align}
where
$B= K_n= (\rm{Im} \alpha_n )^c$ and
$C= \rm{Im} \alpha_n = \ker \delta_n$.
Finally, the equation 
\begin{align*}
\psi_{12}^{n+1} \circ \alpha_{n+1}&=
\alpha_n^{-1} \circ \alpha_n 
=-id_{\lambda _{n}}
\end{align*}
in (\ref{eq.1}) shows that ${\rm Im}\alpha_n \simeq \lambda _{n}$, then
\begin{align}
\lambda_n \simeq {\rm Im}~\alpha_n  =  \ker \delta_n \simeq  
\ker \tilde{d}_n^F \slash {\rm Im} \tilde{d}_{n-1}^F
\end{align}
combining this with (\ref{A_1_p_6}).
This shows Lemma \ref{lem1} again.\\

Lemma \ref{lem1} is 
written in this notation as follows.
\begin{lemma}\label{App_lem1}
Let $\mathcal{C}({}_R\mathcal{M})$ be the category of chain complexes of 
finitely generated free modules over $R$.
Fix a (co)chain complex $F \in \mathcal{C}({}_R\mathcal{M})$ with homomorphisms $\tilde{d} _n^F$ as 
\begin{align*}
F=( \cdots \rightarrow  F_{n-1} \xrightarrow{\tilde{d}_{n-1}^F} F_n  \xrightarrow{\tilde{d}_n^F} F_{n+1}  \rightarrow \cdots ), \quad \tilde{d}_{n}^F\tilde{d}_{n-1}^F=0
\end{align*}
for all $n$. 
We suppose that
a (co)chain complex $\lambda \in ob(\mathcal{C}({}_R\mathcal{M}))$ is given by
\begin{align*}
\lambda&=(\cdots \rightarrow  \lambda_{n-1} \xrightarrow{~0~} \lambda_n \xrightarrow{~0~} \lambda_{n+1} \rightarrow \cdots),
\end{align*}
where $0$ is a trivial homomorphism which maps every (co)chain to zero. For all $n$, under the direct decomposition of 
\begin{align*}
F_n=G_n \oplus {\rm Im}~\tilde{d}_{n-1}^F, 
\end{align*}
we consider a morphism $\alpha:=\{\alpha_n:\lambda_n\to 
G_n
%({\rm Im}~\tilde{d}^F _{n-1})^c
\subset F_n\}\in \mathcal{C}({}_R\mathcal{M})(\lambda,F)$ and $Z:={\rm Cone}(\alpha)=\lambda[1]\oplus F$ with morphisms $\{ d_n^Z:Z_n\to Z_{n+1}\}$ such that 
\begin{align*}
d_n^Z=\left[
\begin{array}{ccc}
0&0&0\\
\alpha_{n+1}&0&0\\
0&\delta_n&0
\end{array}\right]:
\left[\begin{array}{c}
\lambda_{n+1}\\
G_n\\
{\rm Im}~ \tilde{d}_{n-1}^F
\end{array}\right]\to
\left[\begin{array}{c}
\lambda_{n+2}\\
G_{n+1}\\
{\rm Im}~ \tilde{d}_{n}^F
\end{array}\right],
\end{align*}
where $\delta_n :G_n \to {\rm Im}~ \tilde{d}_{n}^F$. If ${\rm Cone}(\alpha)\sim 0$, then 
each $\alpha_n$ is injective, and
$\lambda_n$ and the cohomology group $H^n(F)$ are isomorphic:
\begin{align*}
\lambda _n \simeq \frac{\ker \tilde{d}_{n}^F}{{\rm Im}~\tilde{d}_{n-1}^F}=H^n(F)
\end{align*}
for all $n$.
\end{lemma}

% % % % % % % % % % % % % % % % % %

\end{document}